\documentclass[12pt,psamsfonts]{amsart}
\usepackage{amsmath,amsthm,amsfonts,amssymb}
\usepackage{eucal}
\usepackage{graphicx}
\usepackage[all,knot]{xy}
\xyoption{arc}

\newcommand{\CC}{\mathcal{C}}
\newcommand{\ep}{\epsilon}
\newcommand{\om}{\omega}
\newcommand{\lan}{\langle}
\newcommand{\ra}{\rangle}
\newcommand{\ve}{\varepsilon}
\newcommand{\mS}{\mathcal{S}}

\newcommand{\Rep}{{\rm Rep}}
\newcommand{\zero}{{\mathbf 0}}
\newcommand{\la}{\lambda}
\newcommand{\ga}{\gamma}

\newcommand{\g}{\mathfrak{g}}
\newcommand{\ssl}{\mathfrak{sl}}
\newcommand{\so}{\mathfrak{so}}
\newcommand{\N}{\mathbb{N}}
\newcommand{\Q}{\mathbb Q}

\DeclareMathOperator{\FPdim}{FPdim}

\newcommand{\ma}{\mathbf{a}}
\newcommand{\md}{\mathbf{d}}

\DeclareMathOperator{\End}{End}  \DeclareMathOperator{\Res}{Res}
\DeclareMathOperator{\Hom}{Hom}

\newcommand{\one}{\mathbf{1}}
\newcommand{\C}{\mathbb C}

\newcommand{\Z}{\mathbb Z}

\newcommand{\ts}{\tilde{s}}
\newcommand{\ot}{\otimes}
\newcommand{\B}{\mathcal{B}}

\newcommand{\be}{\mathbf{e}}

\numberwithin{equation}{section}

\newtheorem{theorem}{Theorem}[section]

\newtheorem{cor}[theorem]{Corollary}

\newtheorem{lemma}[theorem]{Lemma}

\newtheorem{conj}[theorem]{Conjecture}

\newtheorem{prop}[theorem]{Proposition}
\theoremstyle{definition}

\newtheorem{remark}[theorem]{Remark}

\newtheorem{ex}[theorem]{Example}
\newtheorem{definition}[theorem]{Definition}

\begin{document}
\title[]
{Localization of unitary braid group representations}

\author{Eric C. Rowell}
\email{rowell@math.tamu.edu}
\address{Department of Mathematics\\
    Texas A\&M University \\
    College Station, TX 77843\\
    U.S.A.}

\author{Zhenghan Wang}
\email{zhenghwa@microsoft.com}
\address{Microsoft Station Q\\CNSI Bldg Rm 2237\\
    University of California\\
    Santa Barbara, CA 93106-6105\\
    U.S.A.}

\thanks{The first author is partially supported by NSA grant H98230-10-1-0215, and thanks C. Galindo and S.-M. Hong for useful correspondence.  Part of the research was carried out while the authors were visiting the Center for Advanced Study, Tsinghua University.}

\begin{abstract}
Governed by locality, we explore a connection between unitary braid group representations associated to a unitary $R$-matrix and to a simple object in a unitary braided fusion category.  Unitary $R$-matrices, namely unitary solutions to the Yang-Baxter equation, afford explicitly \emph{local} unitary representations of braid groups.  Inspired by topological quantum computation, we study whether or not it is possible to reassemble the irreducible summands appearing in the unitary braid group representations from a unitary braided fusion category with possibly different positive multiplicities to get representations that are uniformly equivalent to the ones from a unitary $R$-matrix.  Such an equivalence will be called a \emph{localization} of the unitary braid group representations.  We show that the $q=e^{\pi i/6}$ specialization of the unitary Jones representation of the braid groups can be localized by a unitary $9\times 9$ $R$-matrix.  Actually this Jones representation is the first one in a family of theories $(SO(N),2)$ for an odd prime $N>1$, which are conjectured to be localizable.  We formulate several general conjectures and
discuss possible connections to physics and computer science.
\end{abstract}
\maketitle

\section{Introduction}

One of the questions when a topological quantum field theory (TQFT) computational model was proposed \cite{Freedman1} is whether or not there are non-trivial unitary $R$-matrices, i.e., unitary solutions to the constant Yang-Baxter equation.  In particular, are there unitary $R$-matrices (with algebraic entries) whose representations will result in the Jones polynomial of links evaluated at the unitary roots $q=e^{\pm \frac{2\pi i}{r}}, r>3$?  (The question is not well-posed as there are different ways to define link invariants from representations.  We are thinking about the usual methods: weighted trace or determinant.)  The existence of such $R$-matrices would lead to an efficient approximation of the Jones evaluations by the quantum circuit model of quantum computing.  Though an efficient approximation of quantum invariants including Jones evaluations by quantum computing models exists \cite{FKLW,Wang}, the existence of such unitary $R$-matrices is still open.  We will address this question in this paper and explore its various ramifications.

Jones evaluations of links can be obtained from two representations of the braid groups:  the original unitary Jones representation from von Neumman algebras \cite{J86}, and the non-unitary representation from Drinfeld's $SU(2)$ $R$-matrix \cite{Turaev}.  While it is obvious that any $R$-matrix braid group representation is local, the unitary Jones representation cannot be local as in the case of $R$-matrix because the dimensions of the representation spaces for the braid group $\B_n$ are not $k^n$ for some fixed integer $k$.  It follows that any locality in the unitary Jones representation, if it exists, would be hidden.  Jones unitary representations were later subsumed into unitary TQFTs and the gluing formula of TQFTs can be considered as hidden locality of the Jones representations (see Remark \ref{defrems}).  The gluing formula in TQFTs as hidden locality underlies the simulation of TQFTs by quantum computers \cite{FKW}.

Interesting unitary $R$-matrices are very difficult to construct.  This observation reflects a conflict of explicit locality with unitarity in braid group representations.  The only systematic unitary solutions that we know of are permutations \cite{ESS}.  Smooth deformations of such solutions seem to be all trivial.  We believe that any smooth deformation $R(h)$ of the identity $R(0)=id$ on a Hilbert space $V$ with an orthonormal basis $\{e_i\},i=1,\ldots, m$ would be of the form $R(h)(e_i\otimes e_j)=e^{i h H_{ij}(h)} (e_i\otimes e_j)$ for some smooth family of Hermitian matrices $H_{ij}(h)$.

Topological quantum computation (TQC) (see \cite{Wang} for references) stirred great interests in unitary representations of the braid groups, which describe the statistics of quasi-particles in condensed matter physics \cite{DFN}.  In TQC, unitary matrices from braid generators serve as quantum gates to perform computational tasks.  Therefore, locality, unitarity, and universality of braid group representations are all important ingredients for the application to TQC.  The best source of unitary braid representations is unitary braided fusion category (UBFC) theory.  UBFCs can be constructed from any pair $(G,k)$, where $G$ is a semi-simple Lie group and $k\geq 1$ an integer, called a level.  Jones representations correspond to $(SU(2),k), k\geq 1$.  Most such representations lack an explicit locality: the dimensions of the $\B_n$-representation spaces are rarely of dimension $c^m$ for some fixed integer $c$ and $m$ an integer depending on $n$.  As we will show for Jones $SU(2)$ representations, a refined explicit locality can be achieved only for level $k=1,2,4$ after extra copies of the irreducible summands of the Jones representations are used.  As is known \cite{FLWu,FLW}, levels $k=1,2,4$ are also the only levels for which the braiding alone is not universal for TQC.  In fact, the failure of universality of the braiding is quite dramatic: the image of the braid group generates a \emph{finite group}.  Here we encounter another conflict: explicit locality with universality.

In this paper, we study a refined version of explicit locality inspired by TQC.
We conjecture that explicit local unitary TQFTs in the sense below are inherently braiding non-universal for TQC. i.e., such TQFTs only induce braid representations with finite image (called property \textbf{F} see \cite{RSW,propF}).
 More precisely, we conjecture firstly that a unitary $R$-matrix can only produce braid group representations with finite images (modulo the center) if $R$ is of finite order, and secondly, that braid group representations associated with a braided fusion category can be localized if and only if the Frobenius-Perron dimension of the category is integral.  Taken in tandem, the truth of these conjectures would imply another recent conjecture: integral Frobenius-Perron dimension is necessary and sufficient for a braided fusion category to induce only finite braid group images \cite{RSW,propF}.

 Property \textbf{F} TQFTs are also related to exact solvable statistical mechanics models, result in link invariants which are conjectured to have classical interpretations and to be efficiently computable by Turing machines \cite{FZ,jones89,KMM,LRW,R3}.  Property \textbf{F} TQFTs might also be easier to find in real systems and relevant to the non-abelian statistics of extended objects in dimension three \cite{FHNQWW}.  We also conjecture that topological quantum computing models from Property \textbf{F}  TQFTs can be simulated efficiently by Turing machines.

\section{Localization of braid group representations}

The ($n$-strand) braid group $\B_n$ is defined as the group generated by $\sigma_1,\ldots,\sigma_{n-1}$ satisfying:
\begin{equation}\label{farcommutivity}
\sigma_i \sigma_j=\sigma_j\sigma_i,\;\;\; \textrm{if} \;\;\; |i-j|
\geq 2,
\end{equation}
\begin{equation}\label{braidrelation}
\sigma_i \sigma_{i+1}\sigma_i=\sigma_{i+1} \sigma_{i}\sigma_{i+1}.
\end{equation}

Notice that we may identify $\B_{n}$ with the subgroup of $\B_{n+1}$ generated by $\sigma_i$, $1\leq i\leq n-1$.  When necessary we will denote this map by $\iota:\B_{n}\rightarrow \B_{n+1}$, the index $n$ being clear from the context.
We shall be interested in sequences of completely reducible, finite dimensional, unitary representations of the braid groups $\B_n$ that respect these inclusions $\cdots\subset\B_{n-1}\subset\B_n\subset\cdots$.
More precisely:
\begin{definition}\label{seq}
 An indexed family of complex $\B_n$-representations $(\rho_n,V_n)$ is a \emph{sequence of braid representations} if there exist injective algebra homomorphisms $\tau_n:\C\rho_n(\B_n)\rightarrow \C\rho_{n+1}(\B_{n+1})$ such that the following diagram commutes:

$$\xymatrix{ \C\B_n\ar[r]\ar@{^{(}->}[d]^\iota & \C\rho_n(\B_n)\ar@{^{(}->}[d]^{\tau_n} \\ \C\B_{n+1}\ar[r] & \C\rho_{n+1}(\B_{n+1})}.$$
\end{definition}
Generally $\tau_n$ will be some canonical map and will be supressed.  Note, however, that since $\iota(1)=1$ the identity in $\C\rho_{n+1}(\B_{n+1})$ is $\tau_n\circ\rho_n(1)$.  For example, this excludes the family of $\B_n$-representations obtained from the permutation action of $S_n$ on a basis of $\C^n$ from fitting our definition.  Similarly, the Burau (reduced or otherwise) representations are not a sequence of braid representations by our definition.

The first source of sequences of representations is the \emph{Yang-Baxter equation}:

\begin{definition}\label{rmatdef}
An invertible operator $R\in\End(W\ot W)$ is a solution to the (braided) Yang-Baxter equation if the following relation in $\End(W^{\ot 3})$ is satisfied:
\vspace{.1in}
\begin{enumerate}
 \item[(YBE)] $(R\ot I_W)(I_W\ot R)(R\ot I_W)=(I_W\ot R)(R\ot I_W)(I_W\ot R)$.
\end{enumerate}
\vspace{.1in}
Solutions to the Yang-Baxter equation are called \emph{$R$-matrices}, while the pair $(W,R)$ is called a \emph{braided vector space}.
\end{definition}

Such an $R$-matrix gives rise to a sequence of representations $\rho_R$ of the braid groups $\B_{n}$ on $W^{\ot n}$ via $\sigma_i\rightarrow R_i^W$ where
$$R_i^W:=I_W^{\ot i-1}\ot R\ot I_W^{\ot n-i-1}.$$  The most ubiquitous source of $R$-matrices is finite dimensional semisimple quasi-triangular (or braided) Hopf algebras $H$ (see eg. \cite{Kas}).  Braided vector spaces play a key role in the Andruskiewitch-Schneider program \cite{AS} for classifying pointed Hopf algebras.

We would like to point out that these representations $\rho_R$ are explicitly local, in the sense that the representation space is a tensor power and the braid group generators act on adjacent pairs of tensor factors.  Clearly if $R$ is unitary then the representation $\rho_R$ is unitary as well, but unitary $R$-matrices are surprisingly difficult to find (see Section \ref{ybo}).
\label{s:intro}

A second useful source of sequences of braid group representations is finite dimensional quotients of the braid group algebras $K\B_n$ where $K$ is some field.  A typical example are the Temperley-Lieb algebras $TL_n(q)$ for $q$ a variable \cite{jones87}.  These are the quotients of $\C(q)\B_n$ by ideals containing the element $(\sigma_i-q)(\sigma_i+1)\in\C(q)\B_n$.  The ensuing sequence of finite dimensional algebras $\cdots\subset TL_n(q)\subset TL_{n+1}(q)\subset\cdots$ are semisimple.  As semisimple finite dimensional algebras each $TL_n(q)\cong\bigoplus_{i\in I}\End(W_n^{(i)})$ has a standard faithful representation $W_n=\bigoplus_{i\in I} W_n^{(i)}$, which becomes a $\B_n$ representation by restriction.  Taking $q$ to be an $\ell$th root of unity spoils the semisimplicity, but this can be recovered by taking a further quotient $\overline{TL_n(q)}^\ell$ of $\cup_n TL_n(q)$ by the annihilator of a trace-form (see \cite{J86}).  For the particular choices $q=e^{\pm 2\pi i/\ell}$ Jones showed that the $\B_n$ representations are unitarizable.  Observe that one valuable feature of this construction is that the image of $K\B_n$ obviously generates the corresponding quotient algebra so that the representations $W_n^{(i)}$ are \emph{simple} as $\B_n$ representations.  Moreover, it is clear that the $\B_n$ generators $\sigma_i$ have finite order in the image of these representation.  Other examples of this type of construction are the Hecke-algebras and BMW-algebras (see \cite{wenzl88} and \cite{wenzlbcd}).

\emph{Braided fusion categories} (see eg. \cite{ENO}) are a modern marriage of these two sources of braid group representations.  These are fusion categories equipped with a family of natural transformations: for each pair of objects $X,Y$  one has an invertible morphism $c_{X,Y}\in\Hom(X\ot Y,Y\ot X)$ satisfying certain compatibility (hexagon) equations.  In particular the axioms imply that, for any fixed object $X$ and $1\leq i\leq n-1$, defining
$$R_X^i:=Id_X^{\ot i-1}\ot c_{X,X}\ot Id_X^{\ot n-i-1}\in\End(X^{\ot n})$$
induces a homomorphism $\C\B_n\rightarrow \End(X^{\ot n})$ via $\rho_X:\sigma_i\rightarrow R_X^i$.  More to the point, each $\End(X^{\ot n})$-module restricts to a $\B_n$ representation.  Since fusion categories have only finitely many simple objects, we can describe the simple $\End(X^{\ot n})$-modules completely: they are $W_i^n:=\Hom(X_i,X^{\ot n})$ where $X_i$ is a simple (sub)-object of $X^{\ot n}$.  Thus $\bigoplus_{i} W_i^n$ is a faithful $\End(X^{\ot n})$-module as well as for $\rho_X(\C\B_n)$ (although the $W_i^n$ might be reducible as $\B_n$-representations).    If the category in question is unitary (see \cite{TurWen97}) then the $\B_n$ representations are also unitary.  The representations obtained from quotients of Temperley-Lieb, Hecke and BMW algebras at roots of unity can be obtained in this setting by choosing an appropriate object in a braided fusion category.  With some additional assumptions (modularity), unitary braided fusion categories can be used to construct unitary TQFTs which in turn can be used as models for universal TQC \cite{FKW,FLWu}.

It is important to point out that the objects in a braided fusion categories are not required to be vector spaces (and usually are not) so that the representations $\rho_X$ are \emph{not} explicitly local as those of the form $\rho_R$ are.
For application to TQC, ideally we would like to have unitarity, explicit locality, and universality all in a single TQFT.  We conjecture this is impossible, even if we replace universality with the weaker condition that the braid group images are infinite.  On the other hand, representations of braid groups from $\B_n$ quotients and braided fusion categories have some \emph{hidden} locality and can often be unitarized.

Inspired by such considerations, we make the following:

\begin{definition}\label{localdef}
Suppose $(\rho_n,V_n)$ is a sequence of unitary braid representations.  A \emph{localization} of $(\rho_n,V_n)$ is a braided vector space $(W,R)$ such that for all $n\geq 2$:
\begin{enumerate}
 \item[(i)] there exist algebra homomorphisms $\phi_n:\C\rho(\B_n)\rightarrow \End(W^{\ot n})$ such that $\phi_n\circ\rho(b)=\rho_R(b)$ for $b\in\B_n$ and 
\item[(ii)] the induced $\C\rho(\B_n)$-module structure on $W^{\ot n}$ is faithful.
\end{enumerate}
In other words, if $V_n\cong\bigoplus_{i\in J_n}V_n^{(i)}$ as a $\C\B_n$-module with $V_n^{(i)}$ simple then  $W^{\ot n}\cong \bigoplus_{i\in J_n}\mu_n^iV_n^{(i)}$ as a $\C\B_n$-module for some multiplicities $\mu_n^i>0$.

\end{definition}
To further motivate this definition we present the following:

\begin{ex}\label{doubleex}
 Let $G$ be a finite group and consider the quasi-triangular Hopf algebra $H:=DG$, i.e. the Drinfeld double of the group algebra $\C G$.  The category $H$-mod of finite dimensional $H$-modules has the structure of a (modular) braided fusion category (see \cite{BK}).  Consider the sequence of $\B_n$-representations $(\rho_H,\End_H(H^{\ot n}))$ obtained as above for the object $H$ in $H$-mod (here the subscript $H$ indicates $H$-module endomorphisms, that is, endomorphisms in $H$-mod).  We may also construct a sequence of $\B_n$-representations as follows: let $\check{R}\in \End(H\otimes H)$ be the solution to the YBE for the vector space $H$ so that $(H,\check{R})$ is a braided vector space.  With respect to the usual (tensor product) basis for $H\otimes H$ (see \cite{ERW}) $\check{R}$ is a permutation matrix and hence unitary.  Moreover, standard algebraic considerations show that $(H,\check{R})$ is a localization of $(\rho_H,\End_H(H^{\ot n}))$.  We hasten to point out that the main result of \cite{ERW} shows that the braid group images of these representations are
finite groups.
\end{ex}

Notice that we require a single $(W,R)$ uniformly localizing $\rho_n$ for each $n$.  While this definition suits our purposes, one can imagine other more general definitions.  We clarify this definition with a few remarks:
\begin{remark} \label{defrems}
 \begin{itemize}\item The condition $\mu_n^i>0$ forces $W^{\ot n}$ be a faithful $\rho_n(\C\B_n)$-module (notice that the only simple submodules of $W^{\ot n}$ are the simple submodules of $V_n$).  It is possible that some $V_n^{(i)}$ and $V_n^{(j)}$ are isomorphic as $\B_n$-representations for $i\ne j$.  We are interested in the sequence of \emph{algebras} $\C\rho_n(\B_n)$ for $n>1$ in this paper--in particular the combinatorial structures, such as how the simple $\C\B_n$-modules decompose when restricted to $\C\B_{n-1}$.  From other points of view it might be reasonable to  decompose $V_n$ into $\B_n$-representations and allow multiplicities.  We could then consider a (weaker) localization in this setting, but it would in general not be faithful as a $\C\B_n$-module.
 \item This definition still makes sense with the word ``unitary'' left out if we assume that all of the representations are completely reducible.  Such non-unitary localizations would include sequences of braid group representations obtained from any quasi-triangular Hopf algebra as in Example \ref{doubleex}.
\item From the point of view of quasi-triangular \emph{quasi}-Hopf algebras our definition is somewhat restrictive.  For example, can we localize $\B_n$-representations obtained from \emph{twisted} finite group doubles $D^\om G$?  Although it does not preclude the possibility, the most obvious approach fails.  The reason is that for a quasi-triangular, quasi-Hopf algebra $H$, the braid group generator $\sigma_i$ acts on the $H$-module $V^{\ot n}$ by (see \cite[Lemma XV.4.1]{Kas}):

$$A_i^{-1}(I_V^{\ot i-1}\ot c_{V,V}\ot I_V^{n-i-1})A_i$$
where the $A_i$ are non-trivial operators determined by the associator for $H$, and $c_{V,V}$ is the restriction of $\check{R}$ to $V\ot V$.  One might consider a more general \emph{quasi}-localization that incorporates this approach, but this would take us too far afield so we ignore this possibility for the present.
\item A generalized version of the YBE was introduced in \cite{RZWG}.  The idea is that we can consider $R\in\End(W^{\ot k})$ for $k>2$ and some vector space $W$ satisfying:
$$(R\ot Id_m)(Id_m\ot R)(R\ot Id_m)=(Id_m\ot R)(R\ot Id_m)(Id_m\ot R)$$
where $Id_m:=I_W^{\ot m}$ with $m\geq 1$.  Such generalized Yang-Baxter operators do not automatically yield $\B_n$-representations as relation (\ref{farcommutivity}) is not longer automatic.  However, interesting examples exist where (\ref{farcommutivity}) is satisfied so that one still obtains $\B_n$-representations.  For example we are aware of such examples with $\dim(W)=2$, $k=3$ and $m=1$.  Such examples and the corresponding generalized notion of localization will be considered in future papers.
\item There is a weaker notion of localization found in \cite[proof of Theorem 2.2]{FKW} for topological modular functors.  The idea is to embed the representation space (denoted $V(\Sigma)$ in \cite{FKW}) into a space of the form $W:=(\C^p)^{(n-1)}$ via the gluing axiom in two different ways (by an ``even" and ``odd" pair of pants decompositions).  The $F$-matrices give the change of basis between these two embeddings and the action of a braid $\beta$ is simulated on $W$ using these embeddings, the $F$-matrices and the braid group action on the space $\C^p$.  This is the ``hidden locality" to which we referred in the Introduction. Our strict sense of localization gives a much more direct simulation via a single $R$-matrix.
\end{itemize}
\end{remark}

Associated with any object $X$ in a modular category is a link invariant $Inv_X(L)$.  Given an $R$-matrix one may also define link invariants $T_R(L,\mu_i,\alpha,\beta)$ for any \emph{enhancement} $(\mu_i,\alpha,\beta)$ of $R$ (see \cite{turaevR}).  If $(W,R)$ is a localization of $(\rho_X,\End(X^{\ot n}))$ then one expects the invariant $Inv_X(L)$ to be directly related to $T_R(L,\mu_i,\alpha,\beta)$.  For example, the localization of the Jones representations at level 2 described in \cite{FRW} has an enhancement and the relationship between the two invariants is explicitly described.

\section{Unitary Yang-Baxter operators}\label{ybo}
In this section we discuss the braid group images associated with unitary solutions to the YBE.  In case $\dim(W)=2$, all unitary solutions $R$ have been classified in \cite{Dye} (following the complete classification in \cite{Hiet}), up to conjugation by an operator of the form $Q\ot Q$, $Q\in\End(W)$.  All but one of the conjugacy classes are have a \emph{monomial} representative: a matrix that is the product of a diagonal matrix $D$ and a permutation matrix $P$.  The corresponding braid group representations were analyzed in \cite{frankothesis}, and found to have finite image under the assumption that $R=DP$ has finite order.  The only unitary solution not of monomial form is:
$$\frac{1}{\sqrt{2}}\begin{pmatrix}1 & 0 & 0 & 1\\0 & 1& -1 & 0\\0 & 1&1& 0\\-1 & 0 & 0
&1\end{pmatrix}.$$
The corresponding braid group images were studied in \cite{FRW} and found to be finite as well.

As mentioned above, unitary solutions can be obtained from the quasi-triangular Hopf algebras $DG$, and the corresponding braid group images are known to be finite \cite{ERW}.

This (empirical) evidence as well as other considerations to be described later lead us to make the following:
\begin{conj}\label{uybconj}
\begin{enumerate}
\item[(a)]
Suppose $(V,R)$ is a unitary solution to the YBE such that $R$ has finite (projective) order, with corresponding $\B_n$-representations $(\rho_R,V^{\ot n})$.  Then $\rho_R(\B_n)$ is a finite group (projectively).
\item[(b)] Conversely, a sequence of unitary representations $(\rho_n,V_n)$ such that $\rho_n(\sigma_i)$ has finite (projective) order is localizable only if $\rho_n(\B_n)$ is (projectively) finite for all $n$.
\end{enumerate}
\end{conj}
 If the words \emph{unitary} or \emph{finite order} are omitted Conjecture \ref{uybconj} is false.  For example the standard solution to the YBE corresponding to $U_q\ssl_2$:
$$M:=\begin{pmatrix} q&0&0&0\\0&0&1&0\\0&q&q-1&0\\0&0&0&q\end{pmatrix}$$ has finite order
for $q$ a root of unity but is not unitary if $q\not=1$ and generates an infinite group for $q=e^{\pi i/\ell}$ with $\ell>6$.

  Moreover, one may find unitary solutions to the YBE that do not have finite (projective) order.  The following is an example:
$$\begin{pmatrix}
-1&0&0&0&0&0&0&0&0\\0&0&0&1&0&0&0&0&0\\0&0&0&0&0&0&1
&0&0\\0&2\sqrt{6}/5&1/5&0&0&0&0&0&0
\\0&0&0&0&1&0&0&0&0\\0&0&0&0&0&0&0
&1&0\\0&-1/5&2\sqrt{6}/5&0&0&0&0&0&0
\\0&0&0&0&0&1&0&0&0
\\0&0&0&0&0&0&0&0&1\end{pmatrix}.$$

This last matrix corresponds to a braided vector space of \emph{group type} in the terminology of \cite{AS}.

\section{Non-localizable representations}
In the braided fusion category setting Conjecture \ref{uybconj} is closely related to another fairly recent conjecture (see \cite[Conjecture 6.6]{RSW}).  Braided fusion categories are naturally divided into two classes according to the algebraic complexity of their fusion rules.  In detail, one defines the \emph{Frobenius-Perron dimension} $\FPdim(X)$ of an object $X$ in a fusion category $\CC$ to be the largest eigenvalue of the matrix representing the class of $X$ in the left-regular representation of the fusion ring (or Grothendieck semiring).  Then one defines $\FPdim(\CC)=\sum_i (\FPdim(X_i))^2$ where the sum is over all (isomorphism classes of) simple objects in $\CC$.  If $\FPdim(X_i)\in\N$ for all simple $X_i$ then one says $\CC$ is \emph{integral} while if $\FPdim(\CC)\in\N$ then $\CC$ is said to be \emph{weakly integral}.  It is easy to see that weak integrality is equivalent to $\FPdim(X_i)^2\in\N$ for all simple $X_i$ (see \cite{ENO}).  A braided fusion category $\CC$ is said to have \emph{property \textbf{F}} if for each $X\in \CC$ the $\B_n$-representations on $\End(X^{\ot n})$ have finite image for all $n$.  Then Conjecture 6.6 of \cite{RSW} states: \emph{a braided fusion category $\CC$ has property $\textbf{F}$ if, and only if, $\FPdim(X)^2\in\N$ for each simple object $X$}.  Notice that this conjecture has no formulation for general braided tensor categories, for which one has no sensible (real valued) dimension function.  Some recent progress towards this conjecture can be found in \cite{propF,RUMA} and further evidence can be found in \cite{LRW,LR1}.  Combining with Conjecture \ref{uybconj} we make the following:

\begin{conj}\label{localintegral}
Let $X$ be a simple object in a braided fusion category $\CC$. The representations $(\rho_X,\End(X^{\ot n}))$ are localizable if, and only if, $\FPdim(X)^2\in\N$.
\end{conj}
Note that this conjecture is both less and more general than Conjecture \ref{uybconj}: not all sequences $(\rho_n,V_n)$ of $\B_n$-representations can be realized via fusion categories, but we do not require our braided fusion category $\CC$ to be unitary (just as in \cite[Conjecture 6.6]{RSW}).

The main result of this section is to show that one direction of Conjecture \ref{localintegral} under the assumption that the braid group image generates the centralizer algebras $\End(X^{\ot n})$.
To describe the results we will need some graph-theoretic notions associated with sequences of semisimple finite-dimensional $\C$-algebras ordered under inclusion.  See \cite{GHJ} and graph-theoretical references therein for relevant (standard) definitions.

The main tool we employ is the Perron-Frobenius Theorem.  Our approach is fairly standard, see \cite{wenzl88} for arguments of a similar flavor.  Recall that a non-negative $n\times n$ matrix $A$ is \emph{irreducible} if for each $(i,j)$ there exists a $k=k(i,j)$ so that $(A^k)_{i,j}>0$ and $A$ is \emph{primitive} if there exists a universal $m$ so that $A^m$ has positive entries.  We associate to each irreducible matrix $A$ a digraph $\Gamma(A)$ with $n$ vertices labeled by $1,\ldots, n$ and an arc (directed edge) from $i$ to $j$ if $A_{i,j}>0$.  The digraph $\Gamma(A)$ is clearly \emph{strongly connected}, i.e. there is a (directed) path between any two vertices.  Let $k_i$ be the minimal length of a directed path from $i$ to itself in the graph $\Gamma(A)$ and set $p=\gcd(k_1,\cdots,k_n)$.  The integer $p$ is called the \emph{period} of $A$, and we may partition the vertices into $p$ disjoint sets as follows: define $\mathcal{O}_i$ to be the set of vertices $j$ such that there is a path of length $i+tp$ from $1$ to $j$ for some $t$.  Then $\{\mathcal{O}_i\}_{i=0}^{p-1}$ partitions the vertex set.  By reordering the indices $1,\ldots,n$ as $\mathcal{O}_0,\mathcal{O}_1,\ldots,\mathcal{O}_{p-1}$ the adjacency matrix of $\Gamma(A)$ takes the (block permutation) form:

$$A_\Gamma:=\begin{pmatrix} 0 & C_0& \cdots & 0 \\ 0& 0& \ddots & 0\\
\vdots & \vdots& \ddots & C_{p-2} \\ C_{p-1} & 0   &\cdots &0 \end{pmatrix}.$$
notice that the $C_i$ are not necessarily square matrices, but $C_iC_{i+1}$ is well-defined.
\begin{theorem}[\cite{Gant}]\label{pfthm}
Suppose that $A$ is (non-negative and) irreducible of period $p$.  Then:
\begin{enumerate}
\item[(a)]  $A^p$ is block diagonal with primitive blocks $B_0,\cdots,B_{p-1}$.
\item[(b)] $A$ has a real eigenvalue $\lambda$ such that
\begin{enumerate}
\item[(i)] $|\alpha|\leq \lambda$ for all eigenvalues $\alpha$ of $A$ and
\item[(ii)] There is a strictly positive (right) eigenvector $\mathbf{v}$ associated to $\lambda$
\end{enumerate}
\item[(c)] $\lambda^p$ is a simple eigenvalue of each $B_i$ such that $\lambda^p>|\beta|$ for all eigenvalues $\beta$ of $B_i$ and has strictly positive eigenvector $\mathbf{v_i}$.
\item[(d)] Suppose $A$ is primitive (i.e. $p=1$) and $\mathbf{w}$ is a positive left eigenvector for $\lambda$ chosen so that $\mathbf{w}\mathbf{v}=1$ then:
$$\lim_{s\rightarrow\infty}(1/\lambda A)^s=\mathbf{v}\mathbf{w}$$
\end{enumerate}
\end{theorem}
The real positive eigenvalue $\lambda$ (resp. eigenvector $\mathbf{v}$) in the theorem is called the \emph{Perron-Frobenius eigenvalue} (resp. \emph{eigenvector}).

To any sequence of multi-matrix algebras $\mathcal{S}:=\C=A_1\subset\cdots\subset A_j\subset A_{j+1}\subset\cdots$ with the same identity one associates the \emph{Bratteli diagram} which encodes the combinatorial structure of the inclusions.  The Bratteli diagram for a pair $M\subset N$ of multi-matrix algebras is a bipartite digraph $\Gamma$ encoding the decomposition of the simple $N$-modules into simple $M$-modules, and the inclusion matrix $G$ is the adjacency matrix of $\Gamma$.  More precisely, if $N\cong\bigoplus_{j=1}^t \End(V_j)$ and $M\cong\bigoplus_{i=1}^s \End(W_i)$ the inclusion matrix $G$ is an $s\times t$ integer matrix with entries:
$$G_{i,j}=\dim\Hom_M(\Res_M^NV_j,W_i)$$ i.e. the multiplicity of $W_i$ in the restriction of $V_j$ to $M$.
For an example, denote by $M_n(\C)$ the $n\times n$ matrices over $\C$ and let $N=M_4(\C)\oplus M_2(\C)$ and $M\cong\C\oplus\C\oplus M_2(\C)$ embedded in $N$ as matrices of the form:
$$\begin{pmatrix} a & 0 & 0\\ 0 & a&0\\ 0 & 0 & A\end{pmatrix}\oplus \begin{pmatrix} a & 0\\ 0 & b\end{pmatrix}$$
where $a,b\in\C$ and $A\in M_2(\C)$.  Let $V_1$ and $V_2$ be the simple $4$- and $2$-dimensional $N$-modules respectively, and $W_1$, $W_2$ and $W_3$ be the simple $M$-modules of dimension $1$, $1$ and $2$.  Then the Bratteli diagram and corresponding inclusion matrix for $M\subset N$ are:

$$
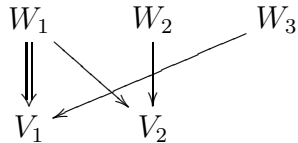
\xymatrix{ W_1\ar@{=>}[d]\ar[dr] & W_2\ar[d] & W_3\ar[dll]\\ V_1 & V_2}$$
and
$$\begin{pmatrix} 2& 1\\ 0 & 1\\ 1 &0\end{pmatrix}.$$

The Bratteli diagram for the sequence $\mathcal{S}$ is the concatenation of the Bratteli diagrams for each pair $(A_k,A_{k+1})$, with corresponding inclusion matrix $G_k$.  We organize this graph into levels (or stories) corresponding to each algebra $A_k$ so that the Bratteli diagram $(A_{k-1},A_k)$ is placed above the vertices labelled by simple $A_k$-modules, and that of $(A_k,A_{k+1})$ is placed below.  (\emph{N.b}: this notion of level is distinct from that of the Introduction mentioned in connection with fusion categories associated with Lie groups.)  Having fixed an order on the simple $A_k$-modules we record the corresponding dimensions in a vector $\mathbf{d}_k$.  Observe that $\mathbf{d}_{k+1}=G_k^T\mathbf{d}_k$.

\begin{definition}
 A sequence of multi-matrix algebras $\mathcal{S}:=\C=A_0\subset\cdots\subset A_n\subset$ can be \emph{combinatorially localized at depth $k$} if there exists an integer $m$ and a sequence of positive integer-valued vectors $\mathbf{a}_n$ so that
\begin{enumerate}
 \item[(a)] $m^n=\lan \ma_n,\md_n\ra$ and
\item[(b)] $m\ma_n=G_n\ma_{n+1}$
\end{enumerate}
for all $n\geq k$.
Here $\md_n$ is the dimension vector for $A_n$ and $G_n$ is the inclusion matrix for $A_n\subset A_{n+1}$.
\end{definition}
Condition (a) is the combinatorial consequence of localization at level $n$ and condition (b) corresponds to compatibility under restriction.  We will call the vectors $\ma_n$ \emph{localization vectors}.  Clearly if a sequence of completely reducible braid group representations $(\rho_n,V_n)$ is localizable then the sequence of algebras $\C\rho_n(\B_n)$ is combinatorially localizable.  Indeed, if $(W,R)$ is a localization of $(\rho_n,V_n)$ as in Definition \ref{localdef} then we may take $m=\dim(W)$ and $\ma_n=(\mu_n^1,\cdots,\mu_n^{|I_n|})$.

Of particular relevance to the fusion category setting are sequences $\mS$ with
 eventually \emph{cyclically $p$-partite} Bratteli diagrams.  That is, for which there exists a level $k$ and a (minimal) period $p\geq 1$ such that for all $n\geq 0$ the subgraphs for $A_{k+np}\subset A_{k+np+1}\subset\cdots\subset A_{k+(n+1)p-1}$ are isomorphic.  If the Bratteli diagram is cyclically $p$-partite then there is an ordering of the simple $A_i$-modules so that $G_{k+t}=G_{k+s}$ if $t\equiv s\mod{p}$.  The combinatorial data for a sequence $\mS$ with cyclically $p$-partite Bratteli diagram can be computed from the dimension vector $\md_k$ for $A_k$ and the matrices $G_k,\ldots,G_{k+p-1}$.  Identifying the $k$ and $k+p$th levels gives a finite (annular) digraph which can be decomposed into its connected components.  These connected components will be strongly connected: they each contain a circuit and no vertex can be terminal (as $A_i\subset A_{i+1}$ implies each $A_i$-module is contained in the restriction of some $A_{i+1}$-module).  If there is only one component (as will be the case for fusion categories) then the associated adjacency matrix will be irreducible with period $p$.

\begin{prop}\label{combprop}
 Suppose $\mS$ is a sequence of multi-matrix algebras with cyclically $p$-partite Bratteli diagram after level $k$.  Moreover, assume that the finite graph obtained by identifying the $k$th and $k+p$th levels is strongly connected.  Let $G^{(i)}=\prod_{j=0}^{p-1} G_{i+k+j}$ be the inclusion matrix of $A_{k+i}$ in $A_{k+i+p}$. Then $\mS$ can be combinatorially localized(at depth $k$) only if,
\begin{enumerate}
 \item[(i)] the Perron-Frobenius eigenvalue $\Lambda_i$ of $G^{(i)}$ is integral for all $i$ and
\item[(ii)] the localization vectors $\ma_{k+i}$ are eigenvectors for $G^{(i)}$ with eigenvalue $\Lambda_i$.
\end{enumerate}
\end{prop}

\begin{proof}
The assumption that the graph is strongly connected implies that the matrices $G^{(i)}$ are primitive and hence the Perron-Frobenius eigenvalue is simple.
We will first show that $\ma_k$ is a Perron-Frobenius eigenvector for $G:=G^{(0)}$.
Assuming that $\mS$ can be combinatorially localized we see that
\begin{equation}\label{pf1}
 m^p\ma_k=G\ma_{k+p}.
\end{equation}
  For ease of notation let us define $\alpha_0:=\ma_k$, $\alpha_n:=\ma_{k+pn}$ and $M=m^p$.  In this notation (\ref{pf1}) implies:
\begin{equation}\label{pf2}
 M^n\alpha_0=G^n\alpha_n
\end{equation}
for all $n\geq 0$.  Let $\Lambda$ denote the Perron-Frobenius eigenvalue of $G$ and let $\mathbf{v}$ and $\mathbf{w}$ be positive right and left eigenvectors respectively, normalized so that $\mathbf{w}\mathbf{v}=1$.  Rewriting (\ref{pf2}) we have:
\begin{equation}\label{pf3}
 \alpha_0=\frac{\Lambda^n}{M^n}\left(\frac{G}{\Lambda}\right)^n \alpha_n.
\end{equation}
Since the left-hand-side of equation (\ref{pf3}) is constant with respect to $n$ we may take a limit and apply the Perron-Frobenius Theorem (for primitive matrices) to obtain:
\begin{equation}
 \alpha_0=\lim_{n\rightarrow\infty}
\frac{\Lambda^n}{M^n}\mathbf{v}\mathbf{w}\alpha_n=\lim_{n\rightarrow\infty}
\left(\frac{\Lambda^n\mathbf{w}\alpha_n}{M^n}\right)\mathbf{v}.
\end{equation}
In particular the scalar part of the limit
$$K:=\lim_{n\rightarrow\infty}
\left(\frac{\Lambda^n\mathbf{w}\alpha_n}{M^n}\right)$$
exists so that $\alpha_0=K\mathbf{v}$.  Thus $\alpha_0=\ma_k$ is a non-zero multiple of $\mathbf{v}$ and hence a Perron-Frobenius eigenvector for $G$.  Similarly, each $\alpha_j$ is also a Perron-Frobenius eigenvector for $G$.  This implies that $\Lambda_0=\Lambda$ is rational since $G$ and $\alpha_0$ are integral.  But $G$ is an integer matrix so its eigenvalues are algebraic integers, which implies that $\Lambda_0$ is a (rational) integer.  The same argument shows that each $\ma_{k+i}$ is a Perron-Frobenius eigenvector for $G^{(i)}$.

\end{proof}

We wish to apply Proposition \ref{combprop} in the setting of (unitary) braided fusion categories (see \cite{BK} for relevant definitions).  In particular,
 let $X$ be a simple object in a braided fusion category $\CC$, and $\CC[X]$ the full braided fusion subcategory generated by $X$ (i.e. the simple objects of $\CC[X]$ are the simple subobjects of $X^{\ot n}$ $n\geq 0$).  That the centralizer algebras $\End(X^{\ot n})$ are semisimple, finite-dimensional $\C$-algebras follows from the axioms.  The simple $\End(X^{\ot n})$-modules are $\Hom(Z,X^{\ot n})$ where $Z$ is a simple object in $\CC[X]$.  The Bratteli diagram for the sequence $$\mS_X:=\End(\one)\cong\C\subset\cdots\subset\End(X^{\ot n})\subset\End(X^{\ot n+1})\subset\cdots$$ (where the inclusion $\End(X^{\ot n})\rightarrow\End(X^{\ot n+1})$ is given by $f\rightarrow f\ot Id_X$) is intimately related to the fusion matrix $N_X$ of $X$ in $\CC[X]$ with entries $(N_X)_{Z,Y}:=\dim\Hom(X\ot Y,Z)$.  We may label the vertices of the Bratteli diagram by pairs $(n,Z)$ where $Z$ is a simple subobject of $X^{\ot n}$, and connect $(n,Z)$ to $(n+1,Y)$ by an $\dim\Hom(X\ot Z, Y)$ arcs.  Now the number of vertices at each level is clearly bounded (by the rank of $\CC[X]$).  Also, the unit object appears in $X^{\ot k}$ for some (minimal) $k$ (see \cite[Appendix F]{DGNO2}), so that if $Y$ appears in $X^{\ot n}$ it appears again in $X^{\ot k+n}$.  This implies that the fusion matrix $N_X$ is irreducible with some period $p$ (dividing $k$) and the associated finite annular digraph $\Gamma(N_X)$ is strongly connected.  For large enough $n$, each inclusion matrix $G_n$ for $\End(X^{\ot n})\subset\End(X^{\ot n+1})$ is the transpose of a submatrix of $N_X$.  Thus the Bratteli diagram for $\mS_X$
  is eventually cyclically $p$-partite with strongly connected annular finite graph.  The Perron-Frobenius eigenvalue of $N_X$ is denoted $\FPdim(X)$.

The examples we have in mind are the (pre-)modular categories associated with quantum groups at roots of unity.  For each simple Lie algebra $\g$ and a $2\ell$th root of unity $q$ one may construct such a category which we denote by $\CC(\g,q,\ell)$.  For details we refer the reader to \cite{BK} and the survey \cite{Rsurvey}.  In the physics literature these categories are often denoted by $(G,k)$ where $G$ is the compact Lie group with Lie algebra $\g$ and $k$ is the level, which is a linear function of the dual Coxeter number of $\g$ and the degree $\ell$ of the root of unity $q^2$.  In particular, the Jones representations of $\B_n$ corresponding to $SU(2)$ at level $k$ may be identified with the representations coming from tensor powers of $X\in\CC(\ssl_2,q,k+2)$ where $X$ is the object analogous to the ``vector representation'' of $\ssl_2$.

The following is a necessary condition for localizability under some additional assumptions.
\begin{theorem}\label{thm:main} Let $X$ be a simple object in a braided fusion category $\CC$ such that
\begin{enumerate}
 \item[(a)]  $\C\rho_n^X(\B_n)=\End(X^{\ot n})$ for all $n\geq 2$ (i.e. the image of $\B_n$ generates $\End(X^{\ot n})$ as an algebra), and
\item[(b)] $(\rho_n^X,\End(X^{\ot n}))$ is localizable
\end{enumerate}
then $\FPdim(X)^2\in\N$.
\end{theorem}
Observe that under these hypotheses, we do not require the representations $\rho_n^X$ to be unitary as complete reducibility of the $\B_n$-representations follows from the semisimplicity of $\End(X^{\ot n})$, by hypothesis (a).

\begin{proof}

It is enough to show that if the sequence of algebras $\cdots\subset\End(X^{\ot n})\subset\End(X^{\ot n+1})\subset\cdots$ is combinatorially localizable then $\FPdim(X)^2\in\N$.
We must verify that in this case the Perron-Frobenius eigenvalue $\Lambda$ in Proposition \ref{combprop} is a power of the Perron-Frobenius eigenvalue $\lambda$
of the fusion matrix $N_X$.  We have already observed that the corresponding Bratteli diagram is cyclically $p$-periodic after some level $k$.  Let $N_X$ be the fusion matrix for $X$ in the subcategory it generates, and choose $p$ and $k$ to be minimal.  Then $N_X$ is irreducible with period $p$ and the corresponding primitive blocks of $N_X^p$ are $G^{(i)}$, so that by Theorem \ref{pfthm} each has $\lambda^p$ as Perron-Frobenius eigenvalue. Now by Proposition \ref{combprop} $\lambda$ must satisfy $x^t-b\in\Z[x]$ so that the minimal polynomial of $\lambda$ is of the form $x^s-c$.  But $\lambda$ must reside in an abelian (cyclotomic) extension of $\Q$ (see \cite[Corollary 8.54]{ENO}).  This implies that $s=1$ or $2$.
\end{proof}

\begin{figure}[t0]\caption{Bratteli diagram for $\CC(\ssl_2,q,5)$}
$\xymatrix{ X\ar[d]\ar[dr]\\  \one\ar[d] & Y\ar[d]\ar[dl] \\ X\ar[d]\ar[dr] & Z\ar[d]\\\one & Y}$
\label{fig:fib}
\end{figure}

\begin{ex}
 We illustrate Theorem \ref{thm:main} with an example corresponding to $\CC(\ssl_2,q,5)$, which has as a subcategory the well-known ``Fibonacci theory."  Here we have just $4$ simple objects $\one,X,Y$ and $Z$ with Bratteli diagram in Figure \ref{fig:fib}.  Moreover, the image of the braid group generates the algebras $\End(X^{\ot n})$ (see, eg. \cite{LZ}).  Observe that the Bratteli diagram is cyclically $2$-partite in this case. The corresponding Perron-Frobenius eigenvalue is $\frac{1+\sqrt{5}}{2}=\FPdim(X)$, so that $(\rho_n^X,\End(X^{\ot n}))$ is not localizable.  The subcategory generated by $Y$ is the (rank $2$) Fibonacci theory.
\end{ex}

\begin{remark}\label{combproprem}
\begin{enumerate}
 \item Observe that under hypotheses (a) and (b) if $p$ is odd then we can conclude that $\FPdim(X)$ is actually an integer.  For example if $X$ is a simple object in a weakly integral modular category $\CC$ such that $X^{\ot 3}\supset\one$ then $\FPdim(X)\in \Z$, whereas in general one can only conclude $\FPdim(X)^2\in\Z$.
\item In a braided fusion category $N_X$ and $N_X^T$ commute, that is, $N_X$ is a normal operator.  This implies that the same is true for any of the blocks $G^{(i)}$ of $(N_X)^p$.  In particular any Perron-Frobenius eigenvector for $G^{(i)}$ is also a
FP-eigenvector for $(G^{(i)})^T$.  But we already know one such eigenvector: the vector of FP-dimensions of the simple subobjects of $X^{\ot i}$.  Thus we see that the localization vectors $\ma_i$, if they exist, must be multiples of the vectors of FP-dimensions.
\item It is expected that condition (a) that the image of $\B_n$ generate the full centralizer algebras $\End(X^{\ot n})$ is superfluous (in the unitary setting).  In general, restricting to the image of $\C\B_n$ gives a refinement of the Bratteli diagram which may fail to be cyclically $p$-partite.  For an easy example of when this occurs consider the symmetric fusion category $\Rep(G)$ for $G$ a finite group possessing a $2$-dimensional irreducible representation $V$.  Then the usual ``flip'' map $P$ provides a localization for the braid group representation on $\End(V^{\ot n})$.  But the number of simple $\C\rho_P(\B_n)$-modules does not stabilize so the Bratteli diagram is not cyclically $p$-partite.
\item  It is known that if $\g$ is of type $A$, $B$ or $C$ and the object $X\in\CC(\g,q,\ell)$ is analogous to the vector representation then the braid group image does generate the centralizer algebras $\End(X^{\ot n})$. This is also known to be the case for the simple object in $X \in\CC(\g_2,q,\ell)$ corresponding to the $7$-dimensional representation of $\g_2$.  For type $A$ this is essentially Jimbo's original quantized version of Schur-Weyl duality between the Hecke algebras and the quantum groups of Lie type $A$.  For types $B$ and $C$ the Hecke algebra is replaced by $BMW$-algebras, see \cite{wenzlbcd}.  For $\g_2$ the result follows from \cite[Theorem 8.5]{LZ}.
The values of $\ell$ for which the categories $\CC(\g,q,\ell)$ fail to have property $\mathbf{F}$ are mostly known, see \cite{J86,FLW,LRW,RUMA}.
\end{enumerate}

\end{remark}

\section{Jones representation at levels $2$ and $4$}
In this section we will show that the (unitary) Jones representations corresponding to levels $2$ and $4$ are localizable.

For level $2$ an explicit localization appears in \cite{FRW}.  In quantum group notation this corresponds to the rank $3$ braided fusion category $\CC(\ssl_2,q,4)$.  The objects are $\one,X$ and $Z$ where $\FPdim(X)=\sqrt{2}$ and $\FPdim(Z)=1$.  The corresponding quantum field theory is closely related to the well-known Ising theory.  The Bratteli diagram is:
$$
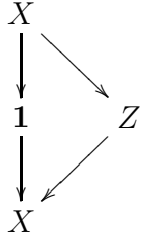
\xymatrix{X\ar[d]\ar[dr]\\ \one\ar[d] & Z\ar[dl]\\ X}$$ and the matrix $$\frac{-e^{-\pi i/4}}{\sqrt{2}}\begin{pmatrix}1 & 0 & 0 & 1\\0 & 1& -1 & 0\\0 & 1&1& 0\\-1 & 0 & 0
&1\end{pmatrix}$$ gives an explicit localization (see \cite[Section 5]{FRW}).

The category $\CC(\ssl_2,q,6)$, corresponding to level $4$, is a rank $5$ category with simple objects $\one,Z$ of dimension $1$, $Y$ of dimension $2$ and $X,X^\prime$ of dimension $\sqrt{3}$.  The fusion rules for this category are determined by:
\begin{eqnarray}
&& X\ot X\cong \one\oplus Y,\quad X\ot X^\prime\cong Z\oplus Y\\
&& X\ot Y\cong X\oplus X^\prime, \quad Z\ot X\cong X^\prime.
\end{eqnarray}
The Bratteli diagram (starting at level 1) is shown in Figure \ref{fig:lis6}, observe that it is cyclically $2$-partite after the third level.
We have $\overline{TL}_n(q^2)\cong\End(X^{\ot n})$ for each $n$, where the isomorphism is induced by
$$g_i\leftrightarrow Id_X^{\ot i-1}\ot c_{X,X}\ot Id_X^{\ot n-i-1}\in\End(X^{\ot n}).$$
Here $c_{X,X}$ is the (categorical) braiding on the object $X$.
The irreducible sectors of $\overline{TL}_n(q^2)$ under this isomorphism are the $\End(X^{\ot n})$-modules $H_{n,W}:=\Hom(W,X^{\ot n})$ where $W$ is one of the $5$ simple objects in $\CC$.  Observe that for $n$ even $W$ must be one of $\one,Y$ or $Z$ while for $n$ odd $W$ is either $X$ or $X^\prime$.  We have the following formulae for the dimensions of these irreducible representations (for $n$ odd):

$$
\dim\Hom(X,X^{\ot n})=\frac{3^{\frac{n-1}{2}}+1}{2},
\quad \dim\Hom(X^\prime,X^{\ot n})=\frac{3^{\frac{n-1}{2}}-1}{2},$$
$$
\dim\Hom(\one,X^{\ot n+1})=\frac{3^{\frac{n-1}{2}}+1}{2}, \quad \dim\Hom(Y,X^{\ot n+1})=3^{\frac{n-1}{2}},
$$$$ \dim\Hom(Z,X^{\ot n+1})=\frac{3^{\frac{n-1}{2}}-1}{2}
$$

\begin{figure}[t0]\caption{Bratteli diagram for $\CC(\ssl_2,q,6)$}
$\xymatrix{ X\ar[d]\ar[dr] & \\ \one\ar[d] & Y\ar[dl]\ar[d]\\X\ar[d]\ar[dr] & X^\prime\ar[d]\ar[dr]\\ \one\ar[d] & Y\ar[dl]\ar[d] & Z\ar[dl]\\X\ar[d]\ar[dr] & X^\prime\ar[d]\ar[dr]\\\one & Y & Z }$
\label{fig:lis6}
\end{figure}

\begin{theorem}\label{loc6}
 The unitary Jones representations corresponding to $TL_n(q^2)$ with $q=e^{\pm\pi i/6}$ can be localized.
\end{theorem}
\begin{proof}
 We will show that the matrix

$R=\frac{i}{\sqrt{3}}\begin{pmatrix}\om&0&0&0&1&0&0&0&\om\\ 0&\om&0&0
&0&\om&1&0&0\\ 0&0&\om&
\om^2&0&0&0&\om^2&0
\\ 0&0&\om^2&\om&0&0&0&\om^2&0
\\ \om&0&0&0&
\om&0&0&0&1\\ 0&1&0
&0&0&\om&\om&0&0\\ 0&
\om&0&0&0&1&\om&0&0
\\ 0&0&\om^2&\om^2&0&0&0&\om&0\\ 1&0
&0&0&\om&0&0&0&\om\end {pmatrix}$

induces a faithful representation $\varphi_n: TL_n(q^2)\rightarrow \End(V^{\ot n})$ ($n\geq 2$) where $V=\C^3$ via $\varphi(g_i)=R_i.$

We first verify that $g_i\rightarrow R_i$ is a unitary representation of the specialized (non-semisimple) Temperley-Lieb algebra $TL_n(q^2)$.  Observe that the matrix $R$ is indeed unitary for either choice of $\omega$, so we fix one and proceed.  It is a computation to see that $R$ satisfies the YBE and has eigenvalues $-1$ (multiplicity 6) and $q=e^{-\pi i/3}$ (multiplicity 3).  We set $E=\frac{R+I}{e^{-\pi i/3}+1}$ so that $E_i=\frac{R_i+I}{e^{-\pi i/3}+1}$ satisfy the defining relations of $TL_n(q^2)$ for $q=e^{-\pi i/6}$:
\begin{eqnarray}
 E_i^2&=&E_i\\
E_iE_{i\pm 1}E_i&=&1/3E_i\\
E_iE_j&=&E_jE_i
\end{eqnarray}
where $|i-j|>1$ as required.  Thus $\varphi_n$ does indeed induce
a representation of $TL_n(q^2)$.

Next we show that $\varphi_n$ actually factors over the semisimple quotient $\overline{TL}_n(q^2):=TL_n(q^2)/(p_5)$ where $p_5$ is the Jones-Wenzl idempotent at the $5$th level.  Observe that $p_5$ generates a $1$ dimensional $\B_5$ representation in $TL_5(q^2)$, with corresponding eigenvalue $-1$.  So it is enough to check that there is no simultaneous eigenvector for $R_1,R_2,R_3$ and $R_4$ (in $3^7$ variables).  A computer calculation verifies this.  Alternatively, one may use the inductive definition of $p_5$ and compute the corresponding $3^7\times 3^7$ matrix and check that it is the $\mathbf{0}$-matrix.  This is computationally more difficult but also checks out.

Finally, we must verify that $\varphi_n$ induces a \emph{faithful} representation of $\overline{TL}_n(q^2)$.  For this it is enough to check that the Jones-Wenzl idempotents $p_i$ do not vanish for $i<5$ by \cite[Prop. A2.1 and Thm. A3.3]{Freedman2}.  This is routine, one simply checks that $R_1,\ldots,R_{i-1}$ do have a simultaneous eigenvector of eigenvalue $-1$. A more direct proof goes as follows: for $n$ odd, $\overline{TL}_n(q^2)$ has $2$ simple modules $W_1(n)$ and $W_2(n)$ of dimensions $\frac{3^{\frac{n-1}{2}}\pm 1}{2}$. So $V^{\ot n}$ must decompose as a $\mu_1 W_1(n)\oplus \mu_2W_2(n)$ for some multiplicities $\mu_i$, since $\varphi_n$ factors over $\overline{TL}_n(q^2)$.  But $\dim(V^{\ot n})=3^n$ so we must have $\mu_i>0$ and $\varphi_n$ is a faithful representation of $\overline{TL}_n(q^2)$ for $n$ odd.  Now the restriction $\varphi_{n-1}$ of $\varphi_n$ contains the restriction of $W_1(n)\oplus W_2(n)$ to $\overline{TL}_{n-1}(q^2)$, which is faithful, so we have the result.
  \end{proof}

With the knowledge that the Jones representations at $q=e^{\pi i/6}$ can be localized,
we can give explicit formulae for the multiplicities of the irreducible sectors of $\overline{TL}_n(q^2)$ appearing in the decomposition of the $3^n$-dimensional representation $(\varphi_n,V^{\ot n})$.  Let us denote by $\mu_{n,W}$ the multiplicity of $H_{n,W}$ in $\varphi_n$.  Then for $k\geq 2$:
\begin{eqnarray*}
 &&\mu_{2k,\one}=3^k,\quad \mu_{2k,Y}=2(3^k), \quad \mu_{2k,Z}=3^k\\
&& \mu_{2k-1,X}=\mu_{2k-1,X^\prime}=3^k
\end{eqnarray*}
While we could give an inductive derivation of these formulas, we instead employ
Proposition \ref{combprop}(ii) (and Remark \ref{combproprem}): the localization vector must be a multiple of the $FP$-eigenvector, i.e. the dimension vector at the given level.  These are: $(1,2,1)$ for $(\one,Y,Z)$ and $(\sqrt{3},\sqrt{3})$ for $(X,X^\prime)$.

We remark that we obtained the localization matrix $R$ in Theorem \ref{loc6} by modifying a solution to the two-parameter version of the YBE found in \cite{Xue}.  With a bit more work we could have obtained these from Jones' \cite{J86} description in terms of automorphisms of $3$-groups.  This approach will be described below in a more general setting.

 The following verifies Conjecture \ref{localintegral} for quantum group categories of Lie type $A_1$, i.e. $\CC(\ssl_2,q,\ell)$.
\begin{cor}\label{cor:nolocal}
 The Jones representation at $q=e^{\pi i/\ell}$ can be localized if and only if $\ell\in\{3,4,6\}$.
\end{cor}
\begin{proof}  The Jones representation at $q=e^{\pi i/\ell}$ corresponds to the sequence of braid group representations on $\End(X^{\ot n})$ where $X$ is the simple object in $\CC(\ssl_2,q,\ell)$ analogous to the vector representation of $\ssl_2$.  Generalized Schur-Weyl duality between the Temperley-Lieb algebras and the quantum group $U_q\ssl_2$ implies that the braid group image generates the centralizer algebras so we may apply Theorem \ref{thm:main}.  The object $X$ has dimension $q+1/q=2\cos(\pi/\ell)$, which is not the square root of an integer except for these values of $\ell$, so a localization is only possible for such values.  Explicit localizations for $\ell=4$ and $\ell=6$ are described above, while the $\ell=3$ case corresponds to a pointed category (i.e. all simple objects are invertible) so is trivially localizable (by a $1\times 1$ $R$-matrix).
\end{proof}
This should be compared with \cite{J86,FLW} from which it follows that the Jones representations are universal for quantum computation if and only if $\ell\not\in\{3,4,6\}$.  Also it is worth observing that the Jones polynomial evaluations at these roots of unity are ``classical'' (see \cite{jones87}) and can be exactly computed in polynomial time (see \cite{R3} for a discussion of this phenomonon).

\section{Generalizations}

The categories $\CC(\ssl_2,q,4)$ and $\CC(\ssl_2,q,6)$ are members of two families of weakly integral, unitary braided fusion categories.  In this section we describe these families of categories which are denoted as $(SO(N),1)$ and $(SO(N),2)$, respectively.  For $N$ odd, they correspond to quantum group categories $\CC(\so_N,q,\ell)$ for $\ell=2N-2$ and $\ell=2N$, respectively, and for $N$ even for $\ell=N-1$ and $N$, respectively.  These level $2$ categories $(SO(N),2)$ have been studied elsewhere, see \cite{Gan} and \cite{propF}.  The corresponding fusion rules also appeared in older work, see \cite{jonessurvey} and \cite{dBG}.

In light of Conjectures \ref{uybconj}, \ref{localintegral} and \cite[Conjecture 6.6]{RSW}, we expect that that the (unitary) braid group representations associated with $(SO(N),1)$ and $(SO(N),2)$ (a) have finite image and (b) can be localized.  We observe that localizations of $(SO(N),1)$ follow from the localization of $\CC(\ssl_2,q,4)$, and describe a possible localization of $(SO(N),2)$ for odd prime $N$ below.

\subsection{$(SO(N),1)$}
For $N=2r+1$ odd, $\CC(\so_N,q,2N-2)$ has rank $3$ with simple objects $\one$, $X_\ep$ and $X_1$ labeled by highest weights $\mathbf{0}$, $\la_r:=\frac{1}{2}(1,\ldots,1)$ and $\la_1:=(1,0,\ldots,0)$ respectively, with corresponding dimensions $1,\sqrt{2}$ and $1$.  The unnormalized $S$-matrix $\tilde{s}$ is:
$$\begin{pmatrix} 1 & \sqrt{2} & 1\\
   \sqrt{2} & 0 & -\sqrt{2}\\ 1 & -\sqrt{2} & 1
  \end{pmatrix}$$ and twist coefficients: $(1,e^{\pi i N/8},-1)$.  The Bratteli diagram is the same as for $\CC(\ssl_2,q,4)$.  Notice that this infinite family consists of only $8$ distinguishable categories depending on $N\pmod{16}$.  All $8$ theories have the same fusion rule as the Ising theory, which is realized by $(SO(17),1)$.  The corresponding rational conformal theories $(SO(N),1)$ have central charges $c=\frac{N}{2}$.

When $N=2r$ is even, $\CC(\so_N,q,N-1)$ has $4$ simple objects labeled by $\mathbf{0}$, $\la_1:=(1,0,\ldots,0)$, $\la_{r-1}:=\frac{1}{2}(1,\ldots,1,-1)$, and $\la_r:=\frac{1}{2}(1,\ldots,1)$.  Each simple object has dimension $1$ and the twists coefficients are: $(1,-1,e^{\pi i N/8},e^{\pi i N/8})$.  The fusion rules are the same as either $\Z_2\times \Z_2$ or $\Z_4$, but the $S$-matrix depends on the value of $N\pmod{16}$, so again there are only $8$ distinguishable categories in this infinite family.

These $16$ different unitary TQFTs realize nicely the $16$-fold way as algebraic models of anyonic quantum systems which encode anyon statistics \cite{Kitaev}.

We make the following:
\begin{definition}
 If the braid representations from any simple object in a unitary modular tensor category can be localized, we will say the corresponding TQFT can be \textit{localized}.
\end{definition}

\begin{theorem}
The $16$ unitary $(SO(N),1)$ TQFTs all have property \textbf{F},  and can be localized.
\end{theorem}

\begin{proof}
 For $N$ odd, the braid group representations are essentially the same as for $\CC(\ssl_2,q,4)$, while for $N$ even, the simple objects are all invertible which implies that the braid group representations are finite abelian groups.  So for $N$ odd the localization is a renormalization of that of the Jones representations at level $2$.  For $N$ even one takes $R$ to be a constant $1\times 1$ matrix to localize.
\end{proof}

\subsection{$(SO(N),2)$}
The families corresponding to $SO(N)$ level $2$ are genuinely infinite but still display similarities depending on the value of $N\pmod{4}$.  Although we do not make explicit use of all of the following data we include it for completeness.

\subsubsection{$N$ odd}
For $N=2r+1$ odd, the category $\CC(\so_N,q,2N)$ has rank $r+4$.
We use the standard labeling convention for the fundamental weights of type $B$:
$\la_1=(1,0,\ldots,0),\ldots,\la_{r-1}=(1,\ldots,1,0)$ and $\la_r=\frac{1}{2}(1,\ldots,1)$.  Observe that the highest root is $\theta=(1,1,0,\ldots,0)$ and $\rho=\frac{1}{2}(2r-1,2r-3,\ldots,3,1)$.
The simple objects are labeled by:
 $$\{\mathbf{0},2\la_1,\la_1,\ldots,\la_{r-1},2\la_r,\la_r,\la_r+\la_1\}.$$
  For notational convenience we will denote by $\ve=\la_r$ and $\ve^\prime=\la_1+\la_r$.  In addition we adopt the following notation from \cite{Gan}: $\la_i=\ga^i$ for $1\leq i\leq r-1$ and $\ga^r=2\la_r$.   The dimensions of the simple objects are:
$\dim(X_{\mathbf{0}})=\dim(X_{2\la_1})=1$, $\dim(X_{\ga^i})=\dim(X_{\la_i})=2$ for $1\leq i\leq r$, and
$\dim(X_{\ve})=\dim(X_{\ve^\prime})=\sqrt{2r+1}$.

Let us denote by $\tilde{s}(\la,\mu)$ the entry of $\tilde{s}$ corresponding to $X_{\la}$ and $X_{\mu}$.  From \cite{Gan} we compute the following:

\begin{eqnarray*}
&& \ts(2\la_1,2\la_1)=1,\quad \ts(2\la_1,\ga^i)=2, \quad \ts(2\la_1,\ve)=\ts(2\la_1,\ve^\prime)=-\sqrt{2r+1}\\
&&\ts(\ga^i,\ga^j)=4\cos(\frac{2ij\pi}{2r+1}),\quad \ts(\ga^i,\ve)=\ts(\ga^i,\ve^\prime)=0\\
&&\ts(\ve,\ve^\prime)=-\ts(\ve,\ve)=\pm\sqrt{2r+1}
\end{eqnarray*}

The remaining entries of $\tilde{s}$ can be determined by the fact that $\tilde{s}$ is symmetric.  The twist coefficients are as follows:
\begin{eqnarray*} &&\theta_{\ga^j}=(-1)^je^{-\pi i j^2/N},\quad 1\leq j\leq r\\ && \theta_{2\la_1}=1,\quad \theta_{\ve}=-\theta_{\ve^\prime}=e^{\pi i (N-1)/8}.
\end{eqnarray*}
The object $X_{\ve}$ generates $\CC(\so_N,q,2N)$, with tensor product decomposition rules:
\begin{enumerate}
\item $X_{\ve}\ot X_{\ve}=X_{\mathbf{0}}\oplus\bigoplus_{i=1}^{r} X_{\ga^i}$
\item $X_{\ve}\ot X_{\ga^i}=X_{\ve}\oplus X_{\ve^\prime}$ for $1\leq i\leq r$
\item $X_{\ve}\ot X_{\ve^\prime}=X_{2\la_1}\oplus\bigoplus_{i=1}^{r}X_{\ga^i}$
\item $X_{\ve}\ot X_{2\la_1}=X_{\ve^\prime}$
\end{enumerate}
The corresponding Bratteli diagram is similar to that of $\CC(\ssl_2,q,6)$ (indeed, identical for $r=1$), except that instead of a single simple object $Y$ of dimension $2$ we have $r$ of them $\{X_{\ga^i}\}$.

\subsubsection{$N$ even}
For $N=2r$ even the category $\CC(\so_N,q,N)$ has rank $r+7$.
The fundamental weights are denoted $\la_1=(1,0,\ldots,0),\ldots\la_{r-2}=(1,\ldots,1,0,0)$, for $1\leq i\leq r-2$
with $\la_{r-1}=\frac{1}{2}(1,\ldots,1,-1)$ and $\la_{r}=\frac{1}{2}(1,\ldots,1)$ the two fundamental spin representations.  We compute the labeling set for $\CC(\so_N,q,N)$ and order them as follows:
$$
\{\mathbf{0},2\la_1,2\la_{r-1},2\la_r,\la_1,\cdots,\la_{r-2},\la_{r-1}+\la_r,\la_{r-1},\la_r,\la_1+\la_{r-1},\la_1+\la_r\}.$$
For notational convenience we will denote by $\ve_1=\la_{r-1}$, $\ve_2=\la_{r}$, $\ve_3=\la_1+\la_{r-1}$
and $\ve_4=\la_1+\la_r$ and set $\ga^j=\la_j$ for $1\leq j\leq r-2$ and $\ga^{r-1}=\la_{r-1}+\la_r$.  In this notation the dimensions of the simple objects are:
$\dim(X_{\ga^j})=2$ for $1\leq i\leq r-1$, $\dim(X_\mathbf{0})=\dim(X_{2\la_1})=\dim(X_{2\la_{r-1}})=\dim(X_{2\la_{r}})=1$ and $\dim(X_{\ve_i})=\sqrt{r}$ for $1\leq i\leq 4$.

The tensor product rules and $\ts$-matrix for $\CC(\so_N,q,N)$ take different forms depending on
the parity of $r$.
In the case that $r$ is odd, one finds that $X_{\ve_1}$ generates $\CC(\so_N,q,N)$.  All simple objects are self-dual (i.e. $X\cong X^*$) except for $X_{\ve_i}$ $1\leq i\leq 4$, $X_{2\la_{r-1}}$ and $X_{2\la_{r}}$.  The corresponding Bratteli diagram for $r$ odd has the same form as that of $\CC(\so_N,q,2N)$ with $N$ odd except that it is (eventually) cyclically $4$-partite.  

\begin{remark}\label{dualrem} 
We expect that this relationship can be extended to isomorphisms between centralizer algebras and braid group representations (possibly for different values of $q$).  The case $N=3$ is essentially known to experts--we sketch how this might be achieved in principal: first identify $(SO(3),2)$ with $(SU(2),4)$ and then employ rank-level duality to obtain $(SU(4),2)$ which is then identified with $(SO(6),2)$ under the usual $A_3$-$D_3$ Lie algebra homomorphism.  For $N=5$ this relationship may be established using the (semisimple quotients of the) $BMW$-algebras $\CC_m(-i, e^{\pi i/10})$ and the modular categories $\CC(\mathfrak{so}_10,Q,10)$ at $Q=e^{7\pi i/10}$.  The details are not particularly enlightening, but the approach is the following: first determine the $\B_3$-reprentation associated with the simple object $X_{\ve_1}$ in $\CC(\mathfrak{so}_{10},Q,10)$ using \cite{TbW}. Check the defining relations for $\CC_m(-i,e^{\pi i/10})$ are satisfied by the images of the $\B_3$ generators and that the values of the trace coincide with those of the uniquely defined trace on $\CC_m(-i,e^{\pi i/10})$.  The uniqueness of the trace and a dimension count establish the desired isomorphism.
\end{remark}

In the case that $r=2m$ is even all objects are self-dual and the subcategory generated by $X_{\ve_1}$ has $m+5$ simple objects
labelled by:
$$\{\zero,2\la_1,2\la_{r-1},2\la_{r},\ga^{2},\ga^4,\ldots,\ga^{r-2}, \ve_1,\ve_4\}.$$
Similarly the subcategory generated by $X_{\ve_2}$ has $m+5$ simple objects, and the full category is generated by $\{X_{\ve_1},X_{\ve_2}\}$.  Although  we expect these categories to be localizable and have property $\mathbf{F}$, the way forward is somewhat less clear in this case.

\subsection{Gaussian Representations}

For each $(SO(N),2)$, we would like to produce a unitary $R$ matrix that localizes the braid group representations associated with the spin objects $X_{\ve}$ and $X_{\ve_i}$.  One complication for such a localization is that no description of the centralizer algebras $\End((X_{\ve})^{\ot n})$ as quotients of $\C\B_n$ is known, except for $N=3$ (Temperley-Lieb algebras), $N=5$ ($BMW$-algebras, see \cite{jones89}) and $N=7$ (see \cite{West}).  In \cite{dBG} a similar situation is studied corresponding to the rational conformal field theory associated with the WZW theory corresponding to $(SO(N),2)$ for $N$ odd.  The principal graph and Bratteli diagram for their theory is the same as ours which suggests that the braid group image is faithfully represented by the so-called ``Gaussian'' representation.  For simplicity, we will focus on the cases where $N$ is an odd prime, which we will denote it by $p$.

The \textit{Gaussian representation} is defined as follows (see \cite{jones89}).  For an odd prime $p\geq 3$ let $\om$ be a primitive $p$-th root of unity and consider the $\C$-algebras $ES(\om,n-1)$ given by generators $u_1,\ldots,u_{n-1}$ satisfying the relations:
\begin{eqnarray}
 u_i^p&=&1\label{es1}\\
u_iu_{i+1}&=&\om^{-2}u_{i+1}u_i\label{es2}\\
u_iu_j&=&u_ju_i, \quad |i-j|>1\label{es3}
\end{eqnarray}
 Notice that $ES(\om,n-1)$ has dimension $p^{n-1}$ and is semisimple (since it is a group algebra of the group presented by (6.1)-(6.3) with $\omega$ interpreted as a central element).  Thus the left-regular action of $ES(\om,n-1)$ on itself is a faithful representation.  If we defined $u_i^*=u_i^{-1}$ and $Tr(1)=1$ and $Tr(w)=0$ for a word in the $u_i$ not proportional to $1$ we can put a Hilbert space structure on $ES(\om,n-1)$ in such a way to make the left-regular representation unitary.

The Gaussian representation $\gamma_n:\B_n\rightarrow ES(\om,n-1)$ is defined on generators of $\B_{n}$ by
$$\gamma_n(\sigma_i)=\zeta\sum_{j=0}^{p-1}\om^{j^2}u_i^j$$
where $\zeta$ is a normalization factor ensuring that $\gamma_n(\sigma_i)$ is unitary (regarded as an operator on $ES(\om,n-1)$).  It was shown in \cite{GJ} that the image of the braid group under this representation is a finite group.  In fact, for $n$ odd the analysis in \cite{GJ} shows that, projectively, $\gamma_n(\B_n)$ is isomorphic to the finite simple group $PSp(n-1,\mathbb{F}_p)$.

\begin{lemma}
 The Gaussian representations $\gamma_n:\B_n\rightarrow ES(\om,n-1)$ can be localized.
\end{lemma}
\begin{proof}
We must find a unitary $p^2\times p^2$ matrix $U$ so that $U_i:=Id^{\ot i-1}\ot U\ot Id^{\ot n-i-1}$ satisfy (\ref{es1}-\ref{es3}).  For this let $\{\be_i\}_{i=1}^p$ denote the standard basis for $\C^p$ and define
\begin{equation}\label{Udef}
U(\be_i\ot\be_j)=\om^{i-j}\be_{i+1}\ot\be_{j+1}
\end{equation}
where the indices on $\be_{i}$ are to be taken modulo $p$.
It is straightforward to check that $U^p=I$ and $U^*=U^{-1}$ (conjugate-transpose).  It is clear that $U_i$ and $U_j$ commute if $|i-j|>1$.  All that remains then is to check (\ref{es2}).  For this it is enough to consider $i=1$:
\begin{eqnarray*}
 &&U_1U_2(\be_i\ot\be_j\ot\be_k)=\om^{i-k-1}\be_{i+1}\ot\be_{j+1}\ot\be_{k+1}=\\
&&\om^{-2}\om^{i-k+1}\be_{i+1}\ot\be_{j+1}\ot\be_{k+1}=\om^{-2}U_2U_1(\be_i\ot\be_j\ot\be_k)
\end{eqnarray*}
so $u_i\rightarrow U_i$ does give a representation of $ES(\om,n-1)$.  A standard trace argument shows that this representation is faithful.  Thus defining
$$R=\zeta\sum_{j=0}^{p-1}\om^{j^2}U^j$$ gives an $R$-matrix localizing the Gaussian representation.
\end{proof}

\begin{conj}
Let $X$ be either the simple spin object $X_{\ve}$ in (the unitary modular tensor category) $(SO(p),2)$ or the simple spin object $X_{\ve_i}$ in $(SO(2p),2)$ for an odd prime $p\geq 3$.  Then the representations $(\rho^X_n,\End(X^{\ot n}))$ are equivalent to the Gaussian representations.  It would follow that, for an odd prime $p\geq 3$, the $(SO(p),2)$ and $(SO(2p),2)$ TQFTs have property \textbf{F} and could be localized in a weak sense.
\end{conj}

\begin{remark}
 This conjecture holds for $p=3$ and $5$, and for $(SO(7),2)$.  For $p=3$ and $p=5$ this goes back to Jones' papers \cite{J86,jones89} (see Remark \ref{dualrem} for the $SO(2p)$ cases).  For $(SO(7),2)$ one must rely upon the description of centralizer algebras $\End(X_\ve^{\ot n})$ for $U_q\so_7$ in \cite{West}.  This description is in terms of generators $K_i,H_i$ and $U_i$ where $1\leq i\leq n-1$ and some 34 sets of relations.  Fortunately these relations need only be verified for $n=3$, but it is a monumentously tedious calculation to check that the Gaussian representation does indeed satisfy these relations for $q=e^{\pi i/14}$.  In terms of the generators $u_i$ defining the Gaussian representation one takes: $$U_i=\frac{1}{\sqrt{7}}\sum_{j=0}^6(q^4u_i)^j,\quad K_i=\frac{[2]^2[3]}{7[3]}\sum_{j=0}^6 (u_i)^j$$
and $R_i=\frac{1}{\sqrt{-7}}\sum_{j=0}^6 q^{4j^2}(u_i)^j$ represents $\sigma_i$. The image of $H_i$ can be determined from the above as: 
$$H_i=q^{-2} +q^5(q+q^{-1})R_i+q^3U_i-qK_i.$$
\end{remark}

We would like to localize braid representations from $(SO(p),2)$ and $(SO(2p),2)$ in the stronger sense as for Jones representations at level $2$ and $4$ (see Remark \ref{defrems} after Definition \ref{localdef}).  The conjecture does not lead to such a localization because the braid group representations associated with the spin objects $X_{\ve}$ and $X_{\ve_i}$ would be equivalent to the Gaussian representations as abstract $\B_n$-representations.  The stronger localization could be achieved if explicit $F$-matrices for $(SO(p), 2)$ and $SO(2p),2)$ were known.

We believe that the braid group representations $(\rho^X_n,\End(X^{\ot n}))$ associated with the spin objects $X_{\ve}$ and $X_{\ve_i}$ are determined by knowing the eigenvalues of a single braid generator $\sigma_i$ as they are all conjugate to each other.  As explained in Chap. $2$ of \cite{Wang}, there are two bases for such representations, $\{e_i^{odd}\}$ and $\{e_i^{even}\}$ such that all odd and even braid generators $\{\sigma_j\}, j=odd$ and $j=even$ are diagonal on $\{e_i^{odd}\}$ and $\{e_i^{even}\}$, respectively.  There is a change of basis matrix between the two bases.  When $p$ is prime, the eigenvalues of the braid generators can be calculated for those representations by using the fusion rules and the graphical calculus of $\{e_i^{odd}\}$ and $\{e_i^{even}\}$.  The eigenvalues of braid generators for the Gaussian representations can be obtained from \cite{jones89,dBG}.  We could then compare the two representations with each other.  We leave this for a future publication.

\section{Applications}

The sequence of unitary TQFTs $(SO(N),1) \cong (SO(N+16),1)$ with topological central charge $c=\frac{N}{2} \;mod\; 8$ has been studied extensively in condensed matter physics as model anyon theories.  The theory $(SO(3),1) \cong (SU(2),2)$ is a leading candidate for the description of anyon statistics in fractional quantum Hall (FQH) liquid at filling fraction $\nu=5/2$ \cite{DFN}.

As observed in \cite{jones89}, the Gaussian representations can be Baxterized by the Fateev-Zamolochikov models in statistics mechanics \cite{FZ,KMM}.  It has been shown that all resulting link invariants from $(SO(N),2)$ TQFTs have definitions in classical topology and can be computed efficiently by Turing machines \cite{jones89,KMM}.  These results provide evidence to our conjectures.

The theory $(SO(3),2) \cong (SU(2),4)$ has been proposed as a description of the FQH liquid at filling fraction $\nu=8/3$ \cite{BW}.  We believe a similar proposal can be made for $(SO(p),2)$ for filling fraction $\nu=2+2/p$ when $p$ is prime.  Of particular interests is the $p=5$ case due to the existence of the FQH liquid at $\nu=12/5$.

Non-abelian statistics of $3$-dimensional extended objects are proposed for a projective representation of the ribbon permutation groups in \cite{FHNQWW}.  Ribbon permutation groups are $3$-dimensional remnants of the braid groups.  They play a role as braid groups does for anyon statistics in dimension two.  We may ask which anyon theory potentially would have a $3$-dimensional ghost image.  The ribbon permutation groups correspond to the Ising anyon theory.  From the analysis in \cite{FHNQWW}, groups that generalize ribbon permutation groups must be finite.  Therefore, Fibonacci theory cannot have a $3$-dimensional ghost image in this sense.  We would speculate that only Property \textbf{F} TQFTs are potential candidates for $3$-dimensional statistics.  In particular, extensions of extra-special or nearly-extra-special $p$-groups by the permutation groups $S_n$ are potential candidates.  How to realize them as in \cite{FHNQWW} is an interesting question.

It follows from the Gottesman-Knill theorem that topological quantum computing models from Ising theory can be efficiently simulated by Turing machines.  We would conjecture that quantum computing models from Property \textbf{F} TQFTs can all be efficiently simulated by Turing machines.  Are there models from Property TQFTs that can implement the Grover algorithm?  We would conjecture the answer to be no, at least for models based on the Ising theory.


\begin{thebibliography}{AA}

\bibitem[AS]{AS} N. Andruskiewitsch; H.-J. Schneider,\emph{ Pointed Hopf algebras.}  New directions in Hopf algebras,  1--68, Math. Sci. Res. Inst. Publ., 43, Cambridge Univ. Press, Cambridge, 2002.









\bibitem[BK]{BK} B.\ Bakalov; A.\ Kirillov, Jr., {\em Lectures on
Tensor Categories and Modular Functors}, University Lecture
Series, vol.\ {\bf 21},  Amer.\ Math.\ Soc., 2001.

\bibitem[BW]{BW}M.\ Barkeshli; X.-G.\ Wen, {\em
Anyon Condensation and Continuous Topological Phase Transitions in Non-Abelian Fractional Quantum Hall States,}
arXiv:1007.2030.

\bibitem[ESS]{ESS}P.\ Etingof; T.\ Schedler; A.\ Soloviev, {\em Set-theoretical solutions to the quantum Yang-Baxter equation.} Duke Math. J. 100 (1999), no. 2, 169--209.


\bibitem[NSSFD]{DFN}C.\ Nayak; S.\ H.\ Simon; A.\ Stern; M.\ Freedman; S.\ Das Sarma,
\emph{Non-Abelian Anyons and Topological Quantum Computation,} Rev. Mod. Phys. \textbf{80} (2008) no. 3, 1083--1159.

\bibitem[dBG]{dBG} J.\ de Boer; J. \ Goeree, {\em Markov traces and ${\rm II}_1$ factors in conformal field theory}.
Comm. Math. Phys. \textbf{139} (1991), no. 2, 267--304.

\bibitem[DGNO]{DGNO2} V.\ Drinfeld; S.\ Gelaki; D.\ Nikshych; V.\ Ostrik, \emph{On braided fusion categories I} arXiv:0906.0620.

\bibitem[D]{Dye} H. Dye, \emph{Unitary solutions to the Yang-Baxter
equation in dimension four}. Quantum information processing \textbf{2} (2002)
nos. 1-2, 117--150 (2003).

\bibitem[ENO]{ENO}P.\ Etingof; D.\ Nikshych; V.\ Ostrik,  \emph{On fusion
categories}. Ann. of Math. (2) 162 (2005), no. 2, 581--642.

\bibitem[ERW]{ERW} P.\ Etingof; E.\ C.\ Rowell; S.\ J.\ Witherspoon,
\emph{Braid group representations from quantum doubles of finite groups.} Pacific J. Math. \textbf{234}  (2008) no. 1, 33--41.

\bibitem[Fr]{frankothesis} J. Franko, Braid group representations via the Yang Baxter equation. Thesis, Indiana University 2008.

\bibitem[FRW]{FRW}  J.\ Franko; E.\ C.\ Rowell; Z.\ Wang, Extraspecial 2-groups and images of braid group
representations, \emph{J. Knot Theory Ramifications} \textbf{15} (2006) no. 4, 1--15.

\bibitem[F1]{Freedman1}M.\ H. \ Freedman, \emph{P/NP, and the quantum field computer}. Proc. Natl. Acad. Sci. USA \textbf{95} (1998) no. 1, 98--101.

\bibitem[F2]{Freedman2}M.\ H. \ Freedman, \emph{A magnetic model with a possible Chern-Simons phase}. With an
appendix by F. Goodman and H. Wenzl. Comm. Math. Phys. \textbf{234} (2003) no. 1, 129--183.


\bibitem[FHNQWW]{FHNQWW}M.\ Freedman, M.\ B. \ Hastings, C.\ Nayak, X.-L.\ Qi, K.\ Walker, Z.\ Wang
{\em Projective Ribbon Permutation Statistics: a Remnant of non-Abelian Braiding in Higher Dimensions.}
arXiv:1005.0583


\bibitem[FKLW]{FKLW}M.\ Freedman; A.\ Kitaev; M.\ Larsen; Z.\ Wang,
\emph{Topological quantum computation.}
Bull. Amer. Math. Soc.
(N.S.) \textbf{40} (2003) no. 1, 31--38.

\bibitem[FKW]{FKW}M.\ H.\ Freedman; A.\ Kitaev; Z.\ Wang,
\emph{Simulation of topological field theories by quantum
computers.} Comm. Math. Phys. \textbf{227} (2002) no. 3, 587--603.

\bibitem[FLW1]{FLWu} M.\ H.\ Freedman; M.\ J.\ Larsen; Z.\ Wang,
    \emph{A modular functor which is universal for quantum computation.} Comm. Math. Phys. \textbf{227} (2002) no. 3, 605--622.

\bibitem[FLW2]{FLW} M.\ H.\ Freedman; M.\ J.\ Larsen; Z.\ Wang,
    \emph{The two-eigenvalue problem and density of Jones representation of braid groups.}
    Comm.\ Math.\ Phys.\ {228} (2002), 177-199.

\bibitem[FZ]{FZ} V.\ A.\ Fateev; A.\ B.\ Zamolodchikov, {\em
Self-dual solutions of the star-triangle relations in $Z_{N}$-models.}
Phys. Lett. A 92 (1982) no. 1, 37--39.



\bibitem[Gn]{Gan} T.\ Gannon, \emph{The level 2 and 3 modular invariants for the orthogonal algebras}, Canad. J. Math. \textbf{52} (2000) no. 3, 503-521.

\bibitem[Gt]{Gant} F. R. Gantmacher, The Theory of Matrices, v. 2, Chelsea, New
York, 1959.

\bibitem[GJ]{GJ} D.\ M.\ Goldschmidt and V.\ F.\ R.\ Jones,
    \textit{Metaplectic link invariants.}
    Geom.\ Dedicata \textbf{31} (1989)  no.\ 2, 165--191.

 \bibitem[GHJ]{GHJ} F. M.
Goodman, P. de la Harpe, V. F. R. Jones, Coxeter graphs and towers of algebras. Mathematical Sciences Research Institute Publications, 14. Springer-Verlag, New York, 1989. x+288 pp.

\bibitem[H]{Hiet} J. Hietarinta, \emph{All solutions to the constant
quantum Yang-Baxter equation in two dimensions.}  Phys. Lett.
A, \textbf{165} (1992), 2452-52.

\bibitem[J1]{J86} V.\ F.\ R.\ Jones, \emph{Braid groups, Hecke algebras and type ${\rm II}\sb 1$ factors,}Geometric methods in
operator algebras (Kyoto, 1983), 242--273, Pitman Res.\ Notes Math.\ Ser.\ {\bf 123}, Longman Sci. Tech., Harlow, 1986.

\bibitem[J2]{jones87} V.\ F.\ R.\ Jones, \emph{Hecke algebra representations of braid groups and link polynomials.}  Ann. of Math. (2)  \textbf{126}  (1987)  no. 2, 335--388.

\bibitem[J3]{jonessurvey} V.\ F.\ R.\ Jones, Notes on subfactors and statistical mechanics.  Braid group, knot theory and statistical mechanics,  1--25, Adv. Ser. Math. Phys., 9, World Sci. Publ., Teaneck, NJ, 1989.

\bibitem[J4]{jones89} V.\ F.\ R.\ Jones, \emph{On a certain value of the Kauffman polynomial,} Comm. Math. Phys. \textbf{125} (1989), 459-467.

\bibitem[Ks]{Kas} C.\ Kassel, Quantum Groups. Graduate Texts in Mathematics \textbf{155}, Springer-Verlag, New York, 1995.

\bibitem[K1]{Kitaev}A.\ Kitaev, {\em Anyons in an exactly solved model and beyond.} Ann. Physics 321 (2006), no. 1, 2--111.

\bibitem[KMM]{KMM}T.\ Kobayashi; H.\ Murakami; J.\ Murakami, {\em
Cyclotomic invariants for links.}
Proc. Japan Acad. Ser. A Math. Sci. 64 (1988), no. 7, 235--238.


\bibitem[LR]{LR1}  M.\ J.\ Larsen, E.\ C.\ Rowell, \emph{An algebra-level version of a link-polynomial identity of Lickorish}. Math.
Proc. Cambridge Philos. Soc. \textbf{144} no. 3 (2008), 623--638.

\bibitem[LRW]{LRW}  M.\ J.\ Larsen, E.\ C.\ Rowell, and Z.\ Wang:
    \emph{The $N$-eigenvalue problem and two applications.}
    \textit{Int.\ Math.\ Res.\ Not.} \textbf{2005} (2005) no.\ 64, 3987--4018.


\bibitem[LZ]{LZ} G.\ Lehrer; R.\ Zhang, \emph{Strongly multiplicity free modules for Lie algebras and quantum groups}, J. Algebra \textbf{306} (2006), 138--174.

\bibitem[NR]{propF} D.\ Naidu; E.\ C.\ Rowell, \emph{A finiteness property
for braided fusion categories}, to appear in Algebr. Represent. Theory

\bibitem[R1]{Rsurvey} E.\ C.\ Rowell, \emph{From quantum groups to unitary modular tensor categories}
in Contemp.\ Math.\ \textbf{413} (2006), 215--230.

\bibitem[R2]{RUMA} E.\ C.\ Rowell, \emph{Braid representations from quantum groups of exceptional Lie type},
Rev. Un. Mat. Argentina
\textbf{51} (2010) no. 1, 165�175.

\bibitem[R3]{R3} E.\ C.\ Rowell, \emph{Two paradigms for topological quantum computation}. Advances in quantum computation, 165--177, Contemp. Math., 482, Amer. Math. Soc., Providence, RI, 2009.

\bibitem[RSW]{RSW} E.\ Rowell; R.\ Stong; Z.\ Wang, \textit{On classification of modular tensor categories},  Comm. Math. Phys.  \textbf{292}  (2009),  no. 2, 343--389.
\bibitem[RZWG]{RZWG} E.\ C.\ Rowell; Y.\ Zhang; Y.-S.\ Wu; M.-L.\ Ge, \textit{Extraspecial two-groups, generalized Yang-Baxter equations and braiding quantum gates}, {Quantum Inf. Comput.} \textbf{10} (2010) no. 7-8, 0685--0702.

\bibitem[Setal]{Xue} C.\ Sun, G.\ Wang, T.\ Hu, C.\ Zhou, Q.\ Wang and K.\ Xue \emph{The representations of Temperley-Lieb algebra and entanglement in a Yang-Baxter system,} Internat. J. Quant. Inf.
 \textbf{7} (2009) no. 6, 1285--1293.

\bibitem[TbW]{TbW} I. Tuba and H. Wenzl,  \textit{Representations of the braid group $\rm{B}_3$  and of $\rm{SL}(2,\rm{Z})$}. Pacific J. Math. \textbf{197} (2001), no. 2, 491--510.
\bibitem[Tu1]{turaevR} V.\ Turaev, \emph{The Yang-Baxter equation and invaraints of links.} Invent. Math. \textbf{92}
(1988), 527-553.
\bibitem[Tu2]{Turaev}V.\ Turaev, Quantum Invariants of Knots and 3-Manifolds, De Gruyter Studies in
Mathematics, Walter de Gruyter (July 1994).

\bibitem[TW]{TurWen97} V.\ G.\ Turaev; H.\ Wenzl, \emph{Semisimple and modular
tensor categories from link invariants}, Math. Ann. \textbf{309}
(1997), 411-461.

\bibitem[Wa]{Wang}Z.\ Wang, Topological quantum computation. CBMS Regional Conference Series in Mathematics, 112. American Mathematical Society, Providence, RI, 2010. xiv+115 pp.


\bibitem[W1]{wenzl88} H.\ Wenzl, \emph{Hecke algebras of type $A_n$ and subfactors.}  Invent. Math.  \textbf{92}  (1988)  no. 2, 349--383.
\bibitem[W2]{wenzlbcd} H.\ Wenzl,
    \emph{Quantum groups and subfactors of type $B$, $C$, and $D$,}
    Comm. Math. Phys. \textbf{133}  (1990)  no. 2, 383--432.
\bibitem[Wb]{West} B.\ W.\ Westbury, \textit{Invariant tensors for the spin representation of $\mathfrak{so}(7)$}.  Math. Proc. Cambridge Philos. Soc.  \textbf{144} (2008)  no. 1, 217--240.


\end{thebibliography}
\end{document}